\documentclass{amsart}
\usepackage[utf8]{inputenc}
\usepackage[T1]{fontenc}

\usepackage{hyperref}
\usepackage{amsfonts}
\usepackage{eurosym}
\usepackage{amssymb}
\usepackage{dsfont}
\usepackage{color}
\usepackage{graphicx}
\usepackage{mathtools}
\usepackage{tikz}
\usetikzlibrary{cd}

\newcommand{\N}{\mathbb N}

\newcommand{\R}{\mathbb R}

\newtheorem{thm}{Theorem}

\newtheorem{lem}[thm]{Lemma}
\newtheorem{prop}[thm]{Proposition}
\newtheorem{cor}[thm]{Corollary}

\newtheorem{problem}[thm]{Problem}

\title{Lineability of functions in $C(K)$ with specified range}
\date{}
\author[A. Bartoszewicz]{Artur Bartoszewicz}
\address{Faculty of Mathematics and Computer Science, \L \'od\'z University, ul. Stefana Banacha 22, 90-238 \L \'od\'z, Poland}
\email{artur.bartoszewicz@wmii.uni.lodz.pl}

\author[S. G\l \c{a}b]{Szymon G\l \c{a}b}
\address{Institute of Mathematics, \L \'od\'z University of Technology, ul.
W\'olcza\'nska 215, 93-005 \L \'od\'z, Poland}
\email{szymon.glab@p.lodz.pl}

\begin{document}

\maketitle

\begin{abstract}
This paper is inspired by the paper of Leonetti, Russo and Somaglia [\textit{Dense lineability and spaceability in certain subsets of $\ell_\infty$.} Bull. London
Math. Soc., 55: 2283--2303 (2023)] and the lineability problems raised therein. It concerns the properties of $\ell_\infty$ subsets defined by cluster points of sequences. Using the fact that the set of cluster points of a sequence $x$ depends only on its equivalence class in $\ell_\infty/c_0$ and that the quotient space $\ell_\infty/c_0$ is isometrically isomorphic to $C(\beta\N\setminus\N)$, we are able to translate lineability problems from $\ell_\infty$ to $C(\beta\N\setminus\N)$. We prove that for a compact space $K$ with properties similar to those of $\beta\N\setminus\N$, the sets of continuous functions $f$ in $C(K)$ with $\vert\operatorname{rng}(f)\vert=\omega$ and those $f$ with $\vert\operatorname{rng}(f)\vert=\mathfrak c$ contain, up to zero function, an isometric copy of $c_0(\kappa)$ for uncountable cardinal $\kappa$. Specializing those results to some closed subspaces $K$ of $\beta\N\setminus\N$ we are able to generalize known results to their ideal versions.    
\end{abstract}

In the past two decades, there have been many papers exploring the existence of large and rich algebraic structures within linear space subsets, function algebras, and their Cartesian products. This topic has garnered significant attention, resulting in the publication of a monograph \cite{mono} and several surveys \cite{survey1,survey2}. Furthermore, the subject has been classified under Mathematical Subject Classification as 46B87. The name given to problems in this area is \textit{lineability} problems. Numerous sets in function and sequence spaces, which naturally arise in various mathematical branches, were analyzed from this standpoint.

A subset $M$ of a linear space $X$ is called $\kappa$-lineable if $M\cup\{0\}$ contains a linear subspace of dimension $\kappa$; $M$ is called lineable if it is $\omega$-lineable. A subset $M$ of Banach space $X$ is called spaceable if $M\cup\{0\}$ contains infinitely dimensional closed subspace of $X$. A subset $M$ of Banach space $X$ is called densely lineable if $M\cup\{0\}$ contains dense linear subspace of $X$. The density of a topological space is the least cardinality of a dense subset. 

The recent paper \cite{LRS} by Leonetti, Russo and Somaglia is devoted to the existence of linear spaces and closed linear spaces in some subsets of $\ell_\infty$, i.e. the space of real bounded sequences with the supremum norm. The paper contains some improvements of Papathanasiou's results from  \cite{P} and lots of new interesting theorems and problems. The considered subsets of $\ell_\infty$ are of the form $\operatorname{L}(\kappa):=\{x\in\ell_\infty:\operatorname{Lim}(x)$ has cardinality $\kappa\}$, where $\operatorname{Lim}(x)$ denotes the set of cluster points of the sequence $x$. Other results on lineability and algebrability of $\ell_\infty$ subsets with prescribed $\operatorname{Lim}(x)$ can be found in  \cite{BBFG, BG, FMMS} and recent preprint \cite{MP}.

Let us recall the problems posed by Leonetti, Russo and Somaglia:
\begin{itemize}
    \item \cite[Problem 6.2]{LRS} \textit{Does $\operatorname{L}(\omega)\cup\{0\}$ contain a closed non-separable subspace?  The same question could be asked for $\bigcup_{\kappa\leq\omega}\operatorname{L}(\kappa)$. }
    \item \cite[Problem 6.3]{LRS} \textit{Is $\bigcup_{1\leq n<\omega}\operatorname{L}(2n)$ lineable?}
    \item \cite[Problem 6.4]{LRS} \textit{Is $\bigcup_{1\leq n<\omega}\operatorname{L}(2n+1)$ densely lineable in $\ell_\infty$?}
\end{itemize}

In the recent preprint \cite{MP} Menet and Papathanasiou answers to \cite[Problem 6.2]{LRS}. More precisely, they have shown that $\operatorname{L}(\omega)$ contains, up to zero sequence, a closed subspace of density continuum but it does not contains an isometric copy of $\ell_\infty$. They also answer \cite[Problem 6.4]{LRS} by showing that $\bigcup_{1\leq n<\omega}\operatorname{L}(2n+1)$ is not densely lineable. Although these problems have recently been solved, they, as well as other results in \cite{LRS} and \cite{MP}, provide motivation for further investigation. Our approach to the lineability problems in $\ell_\infty$ is different and more general. We offer new perspective of looking at that problems via spaces $C(K)$.

We start from a simple observation that $\operatorname{Lim}(x)$ remains the same for all sequences equal to $x$ modulo $c_0$, that is belonging to the same equivalence class in $\ell_\infty/c_0$. Next is the well-known fact, see for example Semadeni's monograph \cite{S}, that the quotient space $\ell_\infty/c_0$ is isometrically isomorphic to the space $C(\beta\N\setminus\N)$, i.e. the space of continuous functions on the remainder of the \v{C}ech-Stone compactification $\beta\N$ of natural numbers with discrete topology. Instead of looking at $\ell_\infty$ we focus on $C(\beta\N\setminus\N)$, its closed subspaces and other $C(K)$ spaces. 

We generalize the lineability problem of sequences with prescribed cardinality of limit points to its idealized version in the following way. Assume that $\mathcal{I}$ is an ideal on $\N$. We say that a sequence $x\in\ell_\infty$ tends to $t$ with respect to $\mathcal{I}$, or $t=\mathcal{I}$-$\lim x$, if $\{n\in\N:\vert x(n)-t\vert \geq \varepsilon\}\in\mathcal{I}$ for every $\varepsilon>0$. A point $t$ is called $\mathcal{I}$-cluster point
\footnote{There is also the notion of $\mathcal{I}$-limit points. We say that $t$ is a $\mathcal{I}$-limit point of a sequence $x$ if there is a set $E\notin\mathcal{I}$ such that $\lim_{n\in E}x(n)=t$ (here there is an ordinary convergence on subsequence). Each $\mathcal{I}$-limit point is an $\mathcal{I}$-cluster point, but the inverse implication is does not hold for every ideal. If $\mathcal{I}$ is an ideal of finite sets, then $\mathcal{I}$-convergence is an ordinary convergence, and then cluster points and limit points coincide.} 
of $x$, if $\{n\in\N:t-\varepsilon\leq x(n)\leq t+\varepsilon\}\notin\mathcal{I}$ for every $\varepsilon>0$. One can write the similar definition using a filter instead of ideal. By $\operatorname{L}_{\mathcal{I}}(\kappa)$ we denote the set of all $x\in\ell_\infty$ such that $x$ has $\kappa$ many cluster points with respect to ideal $\mathcal{I}$. By $c_0(\mathcal{I})$ we denote all sequences from $\ell_\infty$ which are $\mathcal{I}$-convergent to zero. Then the quotient space $\ell_\infty/c_0(\mathcal{I})$ is isometrically isomorphic to $C(P_\mathcal{I})$ where $P_{\mathcal{I}}$ denotes the closed subspace of $\beta\N\setminus\N$ consisting of those ultrafiters on $\N$ which contain a filter dual to $\mathcal{I}$.  

The paper is organized as follows. In Section \ref{SectionTranslation} we show how to translate lineability problems from $\ell_\infty$ to $C(\beta\N\setminus\N)$. The same lineability results hold for sequences $x\in\ell_\infty$ with $\operatorname{Lim}(x)$ having the given property and for functions $f\in C(\beta\N\setminus\N)$ with $\operatorname{rng}(f)$ having that property. In the following sections we prove lineability and spaceability theorems for functions in $C(K)$ with specified range for some classes of compact spaces that contain $\beta\N\setminus\N$. In Section \ref{SectionCountableRanges} we give a sufficient condition for $K$ such that the set of all function in $C(K)$ with countable infinite ranges contains, up to zero function, an isometric copy of $c_0(\kappa)$ for a given cardinal $\kappa$. In Section \ref{SectionLargeRange} we give a sufficient condition for $K$ such that the set of all function in $C(K)$ with uncountable ranges contains, up to zero function, an isometric copy of  $c_0(\kappa)$. In Section \ref{SectionIdeals} we specialized the previous results to compact spaces $P_{\mathcal{I}}$, and as corollaries we obtain a generalization of some results from \cite{LRS} and \cite{MP}. Finally, in Section \ref{Section2Lineability} we show that the set of all functions with interval ranges in $C(P_{\mathcal{I}})$ contains, up to zero function, an isometric copy of $C(\beta\N\setminus\N)$.

\section{How to translate lineability problems from $\ell_\infty$ to $C(\beta\N\setminus\N)$}\label{SectionTranslation}

Having $x\in\ell_\infty$ by $\hat{x}\in\ell_\infty/c_0$ we denote the equivalence class $\{z\in\ell_\infty:x-z\in c_0\}$ generated by $x$. The norm on $\ell_\infty/c_0$ is given by
\[
\Vert\hat{x}\Vert_{\ell_\infty/c_0}=\inf\{\Vert z\Vert:x-z\in c_0\}.
\]
Recall that $\ell_\infty/c_0$ is isometrically isomorphic to $C(\beta\N\setminus\N)$. The isometric isomorphism $T:\ell_\infty/c_0\to\beta\N\setminus\N$ is defined as follows. For $x\in\ell_\infty$, $T(\hat{x})$ is a function $f_{\hat{x}}:\beta\N\setminus\N\to\R$ such that
\[
f_{\hat{x}}(\mathcal{U})=t\iff\forall \varepsilon>0\;\;\{n\in\N:t-\varepsilon<x(n)<t+\varepsilon\}\in\mathcal{U}
\]
for all ultrafilters $\mathcal{U}\in\beta\N\setminus\N$. Shortly, $f_{\hat{x}}(\mathcal{U})=\mathcal{U}$-$\lim x$.

For infinite $A\subseteq\N$ put $\langle A\rangle:=\{\mathcal{U}\in\beta\N\setminus\N:A\in\mathcal{U}\}$. Then $\langle A\rangle$ is a basic clopen (i.e. closed and open) set in $\beta\N\setminus\N$. Let us observe that $f_{\hat{\mathbf{1}}_{A}}=\mathbf{1}_{\langle A\rangle}$. Indeed
\[
f_{\hat{\mathbf{1}}_{A}}(\mathcal{U})=1\iff\forall\varepsilon>0\;\;\{n\in\N:1-\varepsilon<\mathbf{1}_A(n)<1+\varepsilon\}\in\mathcal{U}\iff A\in\mathcal{U}\iff \mathcal{U}\in\langle A\rangle\iff \mathbf{1}_{\langle A\rangle}(\mathcal{U})=1 
\]
and
\[
f_{\hat{\mathbf{1}}_{A}}(\mathcal{U})=0\iff\forall\varepsilon>0\;\;\{n\in\N:-\varepsilon<\mathbf{1}_A(n)<\varepsilon\}\in\mathcal{U}\iff A^c\in\mathcal{U}\iff A\notin\mathcal{U}\iff \mathcal{U}\notin\langle A\rangle\iff\mathbf{1}_{\langle A\rangle}(\mathcal{U})=0.
\]

Observe that for infinite sets $A,B\subseteq\N$ such that $A\cap B$ is finite, we have $\langle A\rangle\cap \langle B\rangle=\emptyset$. Otherwise there would be a (non-principal) ultrafilter $U\in\beta\N\setminus\N$ which contains both $A$ and $B$. Consequently, $\mathcal{U}$ would contains, as a filter, their finite intersection $A\cap B$. This is impossible since $\mathcal{U}$ is a non-principal.  
\begin{lem}\label{lem2}
Assume that $x_n\to x$ in $\ell_\infty$. Then $\hat{x}_n\to\hat{x}$ in $\ell_\infty/c_0$.    
\end{lem}
\begin{proof}
    From the definition of $\ell_\infty/c_0$-norm we obtain that $\Vert \hat{x}_n-\hat{x}\Vert\leq\Vert x_n-x\Vert$, which implies the assertion. 
\end{proof}

\begin{prop}\label{PropRngAndLim} Let $x\in\ell_\infty$. Then
$\operatorname{rng}(f_{\hat{x}})=\operatorname{Lim}(x)$.
\end{prop}

\begin{proof}
We have
\begin{equation}\label{eqRange}
f_{\hat{x}}(\mathcal{U})=t\iff\forall \varepsilon>0\;\;\{n\in\N:t-\varepsilon<x(n)<t+\varepsilon\}\in\mathcal{U}
\end{equation}
and 
\begin{equation}\label{eqLim}
t\in\operatorname{Lim}(x)\iff\forall\varepsilon>0\;\;\{n\in\N:t-\varepsilon<x(n)<t+\varepsilon\}\text{ is infinite.}
\end{equation}
Since every ultrafilter $\mathcal{U}\in\beta\N\setminus\N$ contains only infinite sets, then by \eqref{eqRange} and \eqref{eqLim} we obtain that $\operatorname{rng}(f_{\hat{x}})\subseteq\operatorname{Lim}(x)$.

Now, let $t\in\operatorname{Lim}(x)$ and put $A_k=\{n\in\N:t-1/k<x(n)<t+1/k\}$. By \eqref{eqLim} each $A_k$ is infinite. Since $A_1\supseteq A_2\supseteq\dots$, the family consisting of all infinite subsets of $\N$ and $A_i$'s has the finite intersection property, and therefore it can be extended to some ultrafilter  $\mathcal{U}$. Then $f_{\hat{x}}(\mathcal{U})=t$ by \eqref{eqRange}.
\end{proof}

\begin{prop}
Let $\mathcal{P}$ be a family of compact subsets of $\R$ such that $\{0\}\notin\mathcal{P}$. Let $\kappa$ be a cardinal number. The set
\[
\{f\in C(\beta\N\setminus\N):\operatorname{rng}(f)\in\mathcal{P}\}
\]
is $\kappa$-lineable if and only if the set
\[
\{x\in\ell_\infty:\operatorname{Lim}(x)\in\mathcal{P}\}
\]
is $\kappa$-lineable.    
\end{prop}

\begin{proof}
Let $B$ be a basis of linear space $X$ such that 
\[
X\subseteq\{f\in C(\beta\N\setminus\N):\operatorname{rng}(f)\in\mathcal{P}\}\cup\{0\}.
\]
Let $g$ and $h$ be distinct elements in $B$. There are $x,y\in\ell_\infty$ such that $g=f_{\hat{x}}$ and $h=f_{\hat{y}}$. Since $x\mapsto f_{\hat{x}}$ is linear, then the range of 
\[
f_{\widehat{\alpha x+\beta y}}=\alpha f_{\hat{x}}+\beta f_{\hat{y}}=\alpha g + \beta h 
\]
is in $\mathcal{P}$. Thus $\alpha x+\beta y\in\mathcal{P}$ as well. 

Now, for every $h\in B$ pick $x_h\in\ell_\infty$ with $\widehat{x_h}=h$. Then the above reasoning shows that $B':=\{x_h:h\in B\}$ is a basis (because $\{0\}\notin\mathcal{P}$) of linear space $X'$ such that
\[
X'\subseteq\{x\in\ell_\infty:\operatorname{Lim}(x)\in\mathcal{P}\}\cup\{0\}.
\]
Since $B$ and $B'$ are equinumerous, $\kappa$-lineability of $\{f\in C(\beta\N\setminus\N):\operatorname{rng}(f)\in\mathcal{P}\}$ implies that of $\{x\in\ell_\infty:\operatorname{Lim}(x)\in\mathcal{P}\}$. 

To prove the opposite implication assume that $B'$ is a basis of linear space $X'$ with $X'\subseteq\{x\in\ell_\infty:\operatorname{Lim}(x)\in\mathcal{P}\}\cup\{0\}$. Then for $\alpha_i\in\R\setminus\{0\}$ and $x_i\in B'$ we have $\operatorname{Lim}(\sum\alpha_ix_i)\in\mathcal{P}$. Thus $\operatorname{rng}(f_{\widehat{\sum\alpha_ix_i}})\in\mathcal{P}$. Therefore
\[
X:=\{f_{\hat{x}}:x\in X'\}\subseteq \{f\in C(\beta\N\setminus\N):\operatorname{rng}(f)\in\mathcal{P}\}\cup\{0\}.
\]
Since $\mathcal{P}$ does not contain $\{0\}$, then
\[
0\neq f_{\widehat{\sum\alpha_ix_i}}=\sum\alpha_i f_{\widehat{x_i}}.
\]
This shows that $B:=\{f_{\hat{x}}:x\in B'\}$ is the basis of $X$. 

Recall that $x\mapsto f_{\hat{x}}$ is not one-to-one on $\ell_\infty$. However, it is on $B'$. To see it suppose that $x,y\in B'$, $x\neq y$ and $f_{\hat{x}}=f_{\hat{y}}$. Then $x-y\in c_0$ which is a linear combination of $x$ and $y$ with set of limit points equal to $\{0\}$. But $\operatorname{rng}(f_{\widehat{x-y}})\in\mathcal{P}$. So $\operatorname{Lim}(x-y)\in\mathcal{P}$. This is a contradiction.
\end{proof}

\begin{prop}\label{PropSpaceability}
Let $\mathcal{P}$ be a family of compact subsets of $\R$ such that $\{0\}\notin\mathcal{P}$. Let $\kappa$ be a cardinal number. If the set
\[
\{f\in C(\beta\N\setminus\N):\operatorname{rng}(f)\in\mathcal{P}\}
\]
contains, up to zero function, closed subspace of density $\kappa$, then the set
\[
\{x\in\ell_\infty:\operatorname{Lim}(x)\in\mathcal{P}\}
\]
also contains, up to zero function, closed subspace of density $\kappa$.
\end{prop}

\begin{proof}
Let $X\subseteq\{f\in C(\beta\N\setminus\N):\operatorname{rng}(f)\in\mathcal{P}\}\cup\{0\}$ be a Banach space of density $\kappa$. For $h\in X$ one can find $x_h\in\ell_\infty$ such that $\widehat{x_h}=h$. Assume that $x_{h_n}\to x_h$ in $\ell_\infty$. By Lemma \ref{lem2} $h_n\to h$. Since $X$ is closed, $h\in X$. Thus $\operatorname{rng}(h)\in\mathcal{P}$, which in turn by Proposition \ref{PropRngAndLim} impiles that $\operatorname{Lim}(x_h)\in\mathcal{P}$.  
\end{proof}
We do not know whether the assertion of Proposition \ref{PropSpaceability} can be reversed. For example there is a closed set $S$ in $\ell_\infty$ such that $\hat{S}=\{\hat{x}:x\in S\}$ is not closed. To see it, let $x_n=(\frac{1}{k}+\frac{1}{n})_{k\in\N}+e_n$ where $e_n$ is a sequence having 1 on $n$-th coordinate and zeros on remaining ones. Then $\Vert x_n-x_m\Vert\geq 1$ in $\ell_\infty$ but $\Vert \widehat{x_n}-\widehat{x_m}\Vert=\frac{1}{n}-\frac{1}{m}$ for $n>m$.

\section{Functions in $C(K)$ with countable ranges}\label{SectionCountableRanges}

The following is a simple observation. 
\begin{lem}\label{lem1}
Let $K$ be a compact Hausdorff space, $U_0,U_1,\dots$ be a pairwise disjoint clopen subsets of $K$. Assume that $g_n\to g$ in $C(K)$ is such that 
\begin{itemize}
    \item[(a)] $g_n$ is constant on $U_i$ for every $n$ and $i$;
    \item[(b)] for every $n$, $g_n$ vanishes on all but finitely many $U_i$'s;
    \item[(c)] $g_n$ vanishes on $K\setminus\bigcup_{i=0}^\infty U_i$.
\end{itemize}
Then
\begin{itemize}
    \item[(i)] $g$ is constant on $U_i$ for every $i$;
    \item[(ii)] $g\upharpoonright U_i\to 0$;
    \item[(iii)] $g$ vanishes on $K\setminus\bigcup_{i=0}^\infty U_i$.
\end{itemize}
\end{lem}

\begin{proof}
(i) is clear. 

Suppose that $g\upharpoonright U_i\nrightarrow 0$. Then there is $\delta>0$ such that the set of those indices $i$ with $\vert g\upharpoonright U_i\vert>\delta$ is infinite. Since $g_n\to g$ in $C(K)$, there is $N$ with 
\begin{equation}\label{eq1}
    \forall n,m\geq N\;\; (\Vert g_n-g\Vert<\frac{\delta}{2})
\end{equation}
There is $M$ such that $g\upharpoonright U_i$ vanishes for $i>M$. There is $i_0>M$ with $\vert g\upharpoonright U_{i_0}\vert>\delta$. We can find $m>N$ such that $\vert g_N\upharpoonright U_{i_0}-g_m\upharpoonright U_{i_0}\vert<\delta/2$. Thus by \eqref{eq1} we have $\vert g_m\upharpoonright U_{i_0}-g\upharpoonright U_{i_0}\vert<\delta/2$, which implies $\vert g\upharpoonright U_{i_0}\vert<\delta$. This leads to a contradiction. It proves (ii).

(iii) follows from the fact that the uniform convergence implies the point-wise convergence. 
\end{proof}

For a compact space $K$ we define 
\[
    \operatorname{L}_{K}(\kappa):=\{f\in C(K):\vert\text{rng}(f)\vert=\kappa\}.
\]
\begin{thm}\label{TheoremOmegaRange}
    Let $K$ be a compact Hausdorff space. Assume that there is a family of $\kappa$ many pairwise disjoint non-empty clopen subsets of $K$. Then the set $\operatorname{L}_{K}(\omega)\cup\{0\}$    
    contains an isometric copy of $c_0(\kappa):=\{x\in\R^\kappa:\{\alpha<\kappa: |x(\alpha)|<\varepsilon\}$ is finite for every $\varepsilon>0\}$.
\end{thm}

\begin{proof} Let $\{U_\alpha^i:\alpha<\kappa, i=0,1,\dots\}$ be a family of pairwise disjoint non-empty clopen subsets of $K$. Put
\[
f_\alpha=\sum_{i=0}^\infty\frac{1}{3^i}\mathbf{1}_{U_\alpha^i}.
\]
Note that $f_\alpha\in C(K)$. To see this take any open interval $I\subseteq\R$. If $0\in I$, then $f_\alpha^{-1}(\R\setminus I)$ is a union of finitely many $U_{\alpha}^i$'s. Therefore $f_\alpha^{-1}(\R\setminus I)$ is clopen, and so is $f_\alpha^{-1}(I)$. If $0\notin I$, then $f_\alpha^{-1}(I)$ is a union of some $U_{\alpha}^i$'s, and therefore it is open subset of $K$.

Assume that $(g_n)$ is a sequence from $\text{span}\{f_\alpha:\alpha<\mathfrak c\}$ which tends to some $g$. Each $g_k$ is a linear combination of $f_\alpha$'s. If $g_k=\sum_{j=1}^m a_j f_{\alpha_j}$, then, for every $i\in\N$, $g_k$ equals $\frac{a_j}{3^i}$ on $U^i_{\alpha_j}$ and vanishes on $K\setminus\bigcup_{j=1}^m U^i_{\alpha_j}$. By Lemma \ref{lem1} there is a countable set $\{\alpha_j:j=0, 1, 2, \dots\}\subseteq\kappa$ such that $g$ is constant on $U^i_{\alpha_j}$, vanishes on $K\setminus\bigcup_{j=1}^\infty U^i_{\alpha_j}$ and  $\lim_{i\to\infty}\Vert g\upharpoonright U_{\alpha_j}^i\Vert =0$. 

Next, note that $\Vert f_\alpha-f_{\beta}\Vert= 1$ iff $\alpha\neq \beta$. To see this take any $x\in U^0_\alpha$. Then $x\notin U^0_\beta$. Thus $f_\alpha(x)=1$ and $f_\beta(x)=0$. On the other hand $f_\alpha$ and $f_\beta$ have disjoint supports and $\Vert f_\alpha\Vert=\Vert f_{\beta}\Vert= 1$. Thus  $\Vert f_\alpha-f_{\beta}\Vert\leq 1$. Clearly, $c_0(\kappa)\ni x\mapsto \sum_{\alpha<\kappa}x(\alpha)f_{\alpha}\in V$ is an isometry.   

if $t$ is in the range of $g$, then, by the construction, so is $t/3^i$. This implies that either the range of $g$ is infinite or $g$ is the zero function. On the other hand the range of $g$ is contained in the closure of the set $\{g\upharpoonright U^i_{\alpha_j}:i,j\in\N\}$. Let $R_i=\{g\upharpoonright U^i_{\alpha_j}: j\in\N\}$. Then $R_i=\frac{1}{3^i}R_0$. Take any $\varepsilon>0$. There is $i_0$ such that $R_i\subseteq(-\varepsilon,\varepsilon)$ for $i\geq i_0$. Moreover, $(-\varepsilon,\varepsilon)$ contains almost all elements of $R_i$ for $i<i_0$. Thus $\{g\upharpoonright U^i_{\alpha_j}:i,j\in\N\}\setminus (-\varepsilon,\varepsilon)$ is finite. This shows that 0 is the only limit point of $\{g\upharpoonright U^i_{\alpha_j}: i,j\in\N\}$. Thus the range of $g$ is countable. 
\end{proof}

\section{Non-separable spaceability of functions in $C(K)$ with large range}\label{SectionLargeRange}

Let $X$ be a non-empty set. A Cantor scheme in $X$ is any family $\{A_s:s\in\{0,1\}^{<\N}\}$ of non-empty subsets of $X$ indexed by finite 0-1 sequences such that 
\begin{itemize}
    \item[(i)] $A_\emptyset=X$, where $\emptyset$ stands for empty or null 0-1 sequence;
    \item[(ii)] $A_{s\hat{\;}i}\subseteq A_s$ for every $s\in \{0,1\}^{<\N}$ and $i=0,1$;
    \item[(iii)] $A_{s\hat{\;}0}\cap A_{s\hat{\;}1}=\emptyset$ for every $s\in \{0,1\}^{<\N}$. 
\end{itemize}

The following is most likely mathematical folklore, but we present its proof for the sake of completeness, as we could not find any reference to it.  
\begin{prop}\label{PropLargeRange}
Let $K$ be a compact space and let $L$ be a compact subset of $\R$. Assume that $K$ admits a Cantor scheme of clopen sets. Then there is $f\in C(K)$ with $f(K)=L$.      
\end{prop}

\begin{proof}
Let $g:\{0,1\}^\N\to L$ be a continuous mapping onto $L$. 
Let $\{U_s:s\in\{0,1\}^{<\N}\}$ be a Cantor scheme of clopen sets in $K$. For any $s\in\{0,1\}^{<\N}$ fix $t_s\in\R$ such that $t_s\in g(\langle s\rangle)$ (for example, $t_s=\min g(\langle s\rangle)$). For $n\in\N$ put
\[
f_n=\sum_{\vert s\vert=n}t_s\mathbf{1}_{U_s}.
\]
Clearly $f_n\in C(K)$. We will show that $(f_n)$ is a Cauchy sequence in $C(K)$ and its limit $f$ has the range equal to $L$. 

Let $\varepsilon>0$. By the continuity of $g$ there is $n$ such that $\operatorname{diam}(g(\langle s\rangle))<\varepsilon$ for every $s\in\{0,1\}^n$. Then for $m>n$ we have
\[
\Vert f_n-f_m\Vert = \max\{\vert t_s-t_{s'}\vert:\vert s\vert=n, \vert s'\vert=m \text{ and }s\preceq s'\}\leq\max_{\vert s\vert=n}\operatorname{diam}(g(\langle s\rangle))<\varepsilon.
\]
Therefore $(f_n)$ tends to some $f$ in $C(K)$. 

Let $t\in\operatorname{rng}(f)$. Fix $x\in K$ with $f(x)=t$. Then for any $\varepsilon>0$ there is $n$ such that $\vert f_n(x)-f(x)\vert<\varepsilon$. But $f_n(x)=t_s$ iff $x\in U_s$. Since $t_s\in L$ and $L$ is closed, then $t\in L$, and consequently $\operatorname{rng}(f)\subseteq L$.

Now, let $t\in L$. There is $\alpha\in\{0,1\}^\N$ with $g(\alpha)=t$. Put $s_n=\alpha\upharpoonright n$. By continuity of $g$ we obtain $\operatorname{diam}(g(\langle s_n\rangle))\to 0$, and therefore $t_{s_n}\to t$. Let $x\in\bigcap_{n\in\N}U_{s_n}$. Then $f_n(x)=t_{s_n}$. Thus $f(x)=t$. Hence $L\subseteq\operatorname{rng}(f)$. Finally $L=\operatorname{rng}(f)$.
\end{proof}

\begin{thm}\label{SpaceabilityFamilyPTheorem}
Let $K$ be a compact space. Let $\{U_\alpha:\alpha<\kappa\}$ be a family of pairwise disjoint non-empty clopen subsets of $K$ such that each $U_\alpha$ admits a Cantor scheme of clopen subsets. Assume that $\mathcal{P}$ is a family of compact subsets of $\R$ such that $\{0\}\notin\mathcal{P}$ and there is $P\in\mathcal{P}$ with the property that $\bigcup_{i\in\N}c_iP\cup\{0\}\in\mathcal{P}$ for every non-zero sequence $(c_n)$ tending to zero. 
Then the set
\[
\{f\in C(K):\operatorname{rng}(f)\in\mathcal{P}\}
\]
contains, up to zero function, an isometric copy of $c_0(\kappa)$. 
\end{thm}

\begin{proof}
By Proposition \ref{PropLargeRange} there is a function $h\in C(U_\alpha)$ with $\operatorname{rng}(h)=P$. We define $f_\alpha\in C(K)$ such that $f_\alpha$ equals $h$ on $U_\alpha$, and zero on the complement of $U_\alpha$. By $V$ we denote $\overline{\operatorname{span}}\{f_\alpha:\alpha<\kappa\}$. Note that $\Vert f_\alpha-f_\beta\Vert=\max\{\vert t\vert: t\in P\}$ if $\alpha\neq\beta$, and therefore $V$ is a Banach space with density $\kappa$. 

Let $(g_n)$ be a sequence of function from $\operatorname{span}\{f_\alpha:\alpha<\kappa\}$ tending to $g$ in $C(K)$. By Lemma \ref{lem1} there is a countable set $\{\alpha_i:i\in\N\}\subseteq\kappa$ of ordinals such that $g$ equals $c_if_{\alpha_i}$ on $U_{\alpha_i}$'s for some $c_i\in\R$, vanishes on $K\setminus\bigcup_{i\in\N}U_{\alpha_i}$, and $\Vert g\upharpoonright U_{\alpha_i}\Vert\to 0$. This implies that $(c_i)$ tends to zero. If $g$ is not the zero function, then $c_i\neq 0$ for some $i$. Moreover, $\operatorname{rng}(g)=\bigcup_{i\in\N}c_i\operatorname{rng}(f_{\alpha_i})$. By the assumption on the family $\mathcal{P}$ we obtain that $\operatorname{rng}(g)\in\mathcal{P}$. 
\end{proof}

\begin{cor}\label{CorollaryLargeRange}
Let $K$ be a compact space. Let $\{U_\alpha:\alpha<\kappa\}$ be a family of pairwise disjoint non-empty clopen subsets of $K$ such that each $U_\alpha$ admits a Cantor scheme of clopen subsets. The following sets contain, up to zero function, an isometric copy of $c_0(\kappa)$:
\begin{itemize}
    \item[(i)] $\{f\in C(K):\operatorname{rng}(f)$ has cardinality continuum$\}$;
    \item[(ii)] $\{f\in C(K):\operatorname{rng}(f)$ is an interval$\}$;
    \item[(iii)] $\{f\in C(K):\operatorname{rng}(f)$ is homeomorphic to the Cantor set$\}$;
\end{itemize}
\end{cor}
\begin{proof} Fix a non-zero sequence $(c_n)$ tending to zero. 



(i) and (ii) Let $P=[0,1]$. Then $\bigcup_{i\in\N}c_i[0,1]$ is an interval.

(iii) Let $P$ be the ternary Cantor set. Then $E:=\bigcup_{i\in\N}c_iP\cup\{0\}$ has measure zero, so it has empty interior, and therefore it is zero-dimensional. Note that $E$  contains at least one scaled copy of the ternary Cantor set. Zero is not isolated in this copy. Thus no point of $c_iP$ is isolated in $E$. Thus $E$ as a compact, dense-in-itself, zero-dimensional space is homeomorphic to the Cantor set. 
\end{proof}

One can ask if there is an uncountable compact space $K$ such that the assumption of  
Theorem \ref{TheoremOmegaRange} is satisfied for some uncountable $\kappa$, but the assertion of  Theorem \ref{SpaceabilityFamilyPTheorem}
is not true. A simple example of such space is $\omega_1+1$ -- the space of all ordinal numbers $\alpha$ such that $0\leq\alpha\leq\omega_1$ with the order topology. Then the family $\{\{\alpha\}:\alpha<\omega_1$ is a successor ordinal$\}$ with
cardinality $\omega_1$ and consists of clopen sets. On the other hand, every continuous function $f:\omega_1+1\to\R$ has a countable range. Thus the assertion of Theorem \ref{SpaceabilityFamilyPTheorem}
does not hold for any $\mathcal{P}$ consisting of uncountable compact subsets of reals.

\section{Ideals on natural numbers}\label{SectionIdeals}

In this section we will show that there are natural examples of compact spaces for which the assumptions of Theorem \ref{TheoremOmegaRange} and Theorem \ref{SpaceabilityFamilyPTheorem} are fulfilled. This spaces are defined using ideals $\mathcal{I}$ on natural numbers. More precisely by $P_\mathcal{I}$ we denote the compact space of all ultrafilters which contain a filter $\mathcal{I}^*$ dual to $\mathcal{I}$. It turns out that if $\mathcal{I}$ has the Baire property, then $P_\mathcal{I}$ fulfills the assumptions of Theorem \ref{TheoremOmegaRange}, and if $\mathcal{I}$ has the hereditary Baire property, then $P_\mathcal{I}$ fulfills the assumptions of Theorem \ref{SpaceabilityFamilyPTheorem}. If $\mathcal{I}=\operatorname{Fin}$ that is the ideal of finite sets, then $P_\mathcal{I}=\beta\N\setminus\N$. Moreover $\operatorname{Fin}$ is an $F_\sigma$ ideal, and all Borel ideals have the hereditary Baire property. Therefore by Corollary \ref{CorollaryLargeRange} and Theorem \ref{TheoremOmegaRange} we obtain that both $\operatorname{L}(\mathfrak c)$ and $\operatorname{L}(\omega)$ contain, up to zero function, closed subspaces of density continuum. 

An ideal $\mathcal{I}$ on $\N$ is a family of subsets of $\N$ which is closed under taking finite unions and subsets, does not contain the whole space $\N$ and contains all singletons. 

Given a sequence $(x_n)\in\ell_\infty$ we say that $t$ is its $\mathcal{I}$-cluster point, or $t\in\operatorname{L}_{\mathcal{I}}(x_n)$, if $\{n\in\N:\vert x_n-t\vert<\varepsilon\}\notin\mathcal{I}$ for every $\varepsilon>0$. It is known that the set of all limit points $\operatorname{L}_{\mathcal{I}}(x_n)$ is non-empty compact subset of real numbers for any bounded sequence $(x_n)$. It is a natural question whether the lineability results for subsets of whose sequences $(x_n)$ in $\ell_\infty$ with prescribed set $\operatorname{L}(x_n)$ can be proved when the set of ideal limit points $\operatorname{L}_{\mathcal{I}}(x_n)$ is prescribed.

By $c_0(\mathcal{I})$ we denote all sequences from $\ell_\infty$ which are $\mathcal{I}$-convergent to zero. Let us note that $\operatorname{L}_{\mathcal{I}}(x-y)=\operatorname{L}_{\mathcal{I}}(x)$ for any $x\in\ell_\infty$ and $y\in c_0(\mathcal{I})$. To see it fix $t\in\operatorname{L}_{\mathcal{I}}(x-y)$ and $\varepsilon>0$. Then $A_\varepsilon:=\{n\in\N:\vert(x_n-y_n)-t\vert<\varepsilon/3\}\notin\mathcal{I}$ and $B_\varepsilon:=\{n\in\N:\vert y_n\vert>\varepsilon/3\}\in\mathcal{I}$. Thus $A_\varepsilon\setminus B_\varepsilon\notin\mathcal{I}$. For any $n\in A_\varepsilon\setminus B_\varepsilon$ we have $\vert x_n-t\vert\leq \vert(x_n-y_n)-t\vert+\vert y_n\vert<\varepsilon$. Therefore $t\in\operatorname{L}(x)$, and consequently $\operatorname{L}_{\mathcal{I}}(x-y)\subseteq\operatorname{L}_{\mathcal{I}}(x)$. Since this is true for any $x\in\ell_\infty$ and $y\in c_0(\mathcal{I})$, we have $\operatorname{L}_{\mathcal{I}}(x)=\operatorname{L}_{\mathcal{I}}((x-y)-(-y))\subseteq\operatorname{L}_{\mathcal{I}}(x-y)$.

An ultrafilter $\mathcal{U}$ on $\N$ is called proper if it does not contain any finite set. Having ideal $\mathcal{I}$, by $\mathcal{I}^*$ we denote its dual filter, that is the family $\{\N\setminus A:A\in\mathcal{I}\}$.   
Let $P_{\mathcal{I}}$ denote the set of all proper ultrafilters on $\N$ containing $\mathcal{I}^*$. It turns out that $P_{\mathcal{I}}$ is a closed subset of compact space $\beta\N\setminus\N$. By $C(P_\mathcal{I})$ we denote the Banach space of all continuous function on $P_\mathcal{I}$ with the supremum norm. In \cite{BGW} it is proved that the quotient space $\ell_\infty/c_0(\mathcal{I})$ is isometrically isomorphic to $C(P_\mathcal{I})$\footnote{More precisely, in \cite{BGW} the Authors considered space $\ell_\infty(\mathcal{I})$ consisting of all $\mathcal{I}$-bounded sequences, i.e. sequences $x$ for which there is $K\in\mathcal{I}^*$ such that $x$ restricted to $K$ is bounded in usual sense. On $\ell_\infty(\mathcal{I})$ was considered a seminorm $\Vert x\Vert_\infty^\mathcal{I}$ defined as $\inf\{\lambda>0:(\exists K\in\mathcal{I}^*)(\forall n\in K)\vert x(n)\vert\leq\lambda\}$. It is immediate that a representative of any equivalence class in $\ell_\infty(\mathcal{I})/c_0(\mathcal{I})$ can be chosen from $\ell_\infty$ and quotient norm on $\ell_\infty(\mathcal{I})/c_0(\mathcal{I})$and $\ell_\infty/c_0(\mathcal{I})$ coincide for equivalence classes given by common representatives. Therefore we identify $\ell_\infty(\mathcal{I})/c_0(\mathcal{I})$and $\ell_\infty/c_0(\mathcal{I})$, and use the latter for simplicity.}. 

The following lemma links the properties of an ideal $\mathcal{I}$ with the properties of a Banach space $C(P_\mathcal{I})$. We say that two subsets $A$ and $B$ of $\N$ are $\mathcal{I}$-almost disjoint if $A\cap B\in\mathcal{I}$.

\begin{lem}\label{IdealAlmostDisjointLemma}
Let $\mathcal{I}$ be an ideal on $\N$. 
\begin{itemize}
    \item[(a)] If $A$ and $B$ subsets of $\N$ are $\mathcal{I}$-almost disjoint, then $\langle A\rangle\cap P_\mathcal{I}$ and $\langle B\rangle\cap P_\mathcal{I}$ are disjoint clopen subsets of $P_\mathcal{I}$.
    \item[(b)] If $\{A_\alpha:\alpha<\kappa\}$ is a family of $\mathcal{I}$-almost disjoint subsets of $\N$, then $\{\langle A_\alpha\rangle\cap P_{\mathcal{I}}:\alpha<\kappa\}$ is a family of pairwise disjoint clopen subsets of $P_{\mathcal{I}}$.    
    \item[(c)] If $\{A_s:s\in\{0,1\}^{<\N}\}$ is a family of subsets of $\N$ such that $A_{\emptyset}=\N$, $A_s\notin\mathcal{I}$, , $A_{s\hat{\;}i}\subset A_s$ ($i=0,1$) and $A_{s\hat{\;}0}\cap A_{s\hat{\;}1}=\emptyset$ for every $s$, then the family $\{\langle A_s\rangle\cap P_\mathcal{I}:s\in\{0,1\}^{<\N}\}$ is a Cantor scheme of clopen subsets of $P_\mathcal{I}$. 
\end{itemize}
\end{lem}

\begin{proof}
To prove (a) suppose that there is an ultrafilter $\mathcal{U}$ in $\langle A\rangle\cap\langle B\rangle\cap P_\mathcal{I}$. Then $A,B\in\mathcal{U}$ and $\mathcal{I}^*\subseteq\mathcal{U}$. Then by the definition of ultrafilter, $A\cap B\in\mathcal{U}$. Since $A$ and $B$ are $\mathcal{I}$-almost disjoint, then $A\cap B\in\mathcal{I}$. But $\mathcal{I}^*\subseteq\mathcal{U}$ implies that $\N\setminus (A\cap B)\in\mathcal{U}$, and consequently the set $A\cap B$ and its complement belong to $\mathcal{U}$ which is a contradiction. 
Parts (b) and (c) follow immediately from (a). 
\end{proof}


Ideals on $\N$ can be viewed (via the mapping $\N\supseteq A\mapsto\mathbf{1}_A\in\{0,1\}^\N$) as subsets of the Cantor space $\{0,1\}^\N$. Therefore they may have the Baire property, be Borel, $F_\sigma$ etc. It is well-known characterization of ideals with the Baire property due to Jalali-Naini \cite{J} and Talagrand \cite{T}: ideal $\mathcal I$ on $\N$ has the Baire property if and only if there is an infinite sequences $n_1<n_2<n_3<\dots$ of indices such that no member of $\mathcal{I}$ contains infinitely many intervals $[n_i,n_{i+1})$ ($=\{n_i,n_i+1,\dots,n_{i+1}-1\}$. Let $X\subseteq\N$ be such that $X\notin\mathcal{I}$. Then $\mathcal{I}\upharpoonright X:=\{A\cap X:A\in\mathcal{I}\}$ is an ideal on a countable infinite set $X$. One can consider the Baire property of $\mathcal{I}\upharpoonright X$ as a subset of $\{0,1\}^X$. We say that $\mathcal{I}$ has the hereditary Baire property, provided $\mathcal{I}\upharpoonright X$ has the Baire property for every $X\notin\mathcal{I}$. Note that Borel ideals have the Baire property, and $\mathcal{I}\upharpoonright X$ is a closed subset of $\mathcal{I}$. Therefore Borel ideals have the hereditary Baire property.

Let us note the following observation (Proposition \ref{PropositionBaire} (a) is a piece of mathematical folklore, but we add its short proof for the sake of completeness, see for example \cite[Lemma 2.3]{L}).  

\begin{prop}\label{PropositionBaire}
Let $\mathcal{I}$ be an ideal on $\N$. 
\begin{itemize}
    \item[(a)] If $\mathcal{I}$ has the Baire property, then there is $\mathcal{I}$-almost disjoint family with cardinality continuum.  
    \item[(b)] If $\mathcal{I}$ has the hereditary Baire property, then for every $X\notin\mathcal{I}$ there is a family $\{A_s:s\in\{0,1\}^{<\N}\}$ of subsets of $X$ such that $A_{\emptyset}=X$, $A_s\notin\mathcal{I}$, $A_{s\hat{\;}i}\subset A_s$ ($i=0,1$) and $A_{s\hat{\;}0}\cap A_{s\hat{\;}1}=\emptyset$ for every $s$.
\end{itemize}
\end{prop}

\begin{proof}
Assume that $\mathcal{I}$ has the Baire property. Let $\mathcal{A}$ be a family of cardinality continuum consisting of almost disjoint subsets of $\N$. Then the family $\{\bigcup_{i\in A}[n_i,n_{i+1}):A\in\mathcal{A}\}$ has cardnality continuum and consists of $\mathcal{I}$-almost disjoint subsets of $\N$. This proves (a). 

One can define by simple induction a family $\{A_s:s\in\{0,1\}^{<\N}\}$ of subsets of $\N$ such that $A_{\emptyset}=\N$, $A_s$ if infinite, $A_{s\hat{\;}i}\subset A_s$ ($i=0,1$) and $A_{s\hat{\;}0}\cap A_{s\hat{\;}1}=\emptyset$ for every $s$. Then the family $\{\bigcup_{i\in A_s}[n_i,n_{i+1}):s\in\{0,1\}^{<\N}\}$ has the desired property for $X=\N$. Using the hereditary Baire property we can repeat the reasoning for every $X\notin\mathcal{I}$ to get (b).
\end{proof}
Now, we are ready to prove the following.
\begin{thm}\label{MainTheorem}
Let $\mathcal{I}$ be an ideal on $\N$. \begin{itemize}
    \item[(i)] If $\mathcal{I}$ has the Baire property, then $\operatorname{L}_{\mathcal{I}}(\omega)$ contains, up to zero function, a closed subspace of density continuum.
    \item[(ii)] If $\mathcal{I}$ has the hereditary Baire property, then $\operatorname{L}_{\mathcal{I}}(\mathfrak c)$ contains, up to zero function, a closed subspace of density continuum.
\end{itemize}
\end{thm}
\begin{proof}
First, note that if you replace $C(\beta\N\setminus\N)$ with $C(P_{\mathcal{I}})$ and $\operatorname{L}(x)$ with $\operatorname{L}_{\mathcal{I}}(x)$, then the results in Section \ref{SectionTranslation} remain true. Then using Proposition \ref{PropositionBaire}, Lemma \ref{IdealAlmostDisjointLemma} and Theorem \ref{TheoremOmegaRange} and Theorem \ref{SpaceabilityFamilyPTheorem}, we obtain the assertion of Theorem \ref{MainTheorem}.        
\end{proof}

Since Borel ideals have the hereditary Baire property and the ideal of finite sets is an $F_\sigma$ ideal, then we obtain the following. 
\begin{cor}\label{CorollaryMainTheorem}
    The following sets contain, up to zero function, a Banach space of density $\mathfrak c$:
    \begin{itemize}
        \item[(i)] $\{x\in\ell_\infty:\operatorname{Lim}(x)$ has cardinality continuum$\}$;
        \item[(ii)] $\{x\in\ell_\infty:\operatorname{Lim}(x)$ is an interval$\}$;
        \item[(iii)] $\{x\in\ell_\infty:\operatorname{Lim}(x)$ is homeomorphic to the Cantor set$\}$
    \end{itemize}
\end{cor}
\begin{proof}
    The ideal of finite sets is an $F_\sigma$ ideal. Therefore is has the hereditary Baire property. Thus (i) follows from Theorem \ref{MainTheorem}. By Proposition \ref{PropositionBaire}, Lemma \ref{IdealAlmostDisjointLemma} and Corollary \ref{CorollaryLargeRange} we obtain (ii) and (iii). 
\end{proof}
Corollary \ref{CorollaryMainTheorem}(ii) should be compared with \cite{FMMS} and Corollary \ref{CorollaryMainTheorem}(iii) with \cite{BG,BBFG} where the Authors considered algebrability of those sets.

It is easy to give an example of an ideal $\mathcal{I}$ which has the Baire property but not the hereditary Baire property. Let $\operatorname{Fin}(2\N-1)$ be an ideal of finite sets on odd numbers, and let $\mathcal{I}_{\max}(2\N)$ be a maximal ideal on even numbers. We define $\mathcal{I}$ as $\{A\cup B:A\in\operatorname{Fin}(2\N-1)$ and $B\in\mathcal{I}_{\max} (2\N)\}$. Take any $n_1<n_2<...$ such that $n_{i+1}-n_i\geq 2$ for every $i$. Then any union $X$ of infinitely many integer intervals $[n_i,n_{i+1})$ contains infinitely odd numbers, and therefore $X\notin\mathcal{I}$. By Jalali-Naini--Talagrand characterization, $\mathcal{I}$ has the Baire property. However, $\mathcal{I}\upharpoonright 2\N=\mathcal{I}_{\max} (2\N)$ and maximal ideals on $2\N$ do not have Baire property in $\{0,1\}^{2\N}$.   Now, consider the space $P_{\mathcal{I}}$. It is a union of two disjoint clopen sets $P_{\mathcal{I}}\cap\langle 2\N\rangle$ and $P_{\mathcal{I}}\cap\langle 2\N-1\rangle$. Note that $P_{\mathcal{I}}\cap\langle 2\N\rangle$ has one element, while $P_{\mathcal{I}}\cap\langle 2\N-1\rangle$ is homeomorphic to $\beta\N\setminus\N$. In general, if there is $A\notin\mathcal{I}$ such that $\mathcal{I}\upharpoonright A$ has the hereditary Baire property, then $\mathcal{I}$ fulfills the assertions of both (i) and (ii) in Theorem \ref{MainTheorem}. 

We have been informed by Jacek Tryba, that in \cite{KST} there is an even more extreme ideal. Let $\mathcal{I}_{\max}$ be a maximal ideal on $\N$. Define $\mathcal{I}$ on $\N\times\N$ by $
A\in\mathcal{I}\iff P_1(A)\in\mathcal{I}_{\max}$
where $P_1(n,k)=n$ is the projection on the first coordinate. In \cite[Proposition 3.15]{KST} it is proved that $\mathcal{I}$ have the Baire property. 
Given any $C\notin\mathcal{I}$ we have $P_1(C)\notin\mathcal{I}_{\max}$. For any $n\in P_1(C)$ find $k_n$ with $(n,k_n)\in C$. Then $B=\{(n,k_n):n\in P_1(C)\}\notin\mathcal{I}$. Then $\mathcal{I}\upharpoonright B$ is an maximal ideal and therefore it does not have the Baire property. We will show that $\mathcal{I}$ still fulfills the assertion of Proposition \ref{PropositionBaire}. Let $\{A_\alpha:\alpha<\mathfrak c\}$ be an almost disjoint family on $\N$. For each $\alpha$ we define $\{A_{\alpha,s}:s\in\{0,1\}^{<\N}\}$ of subsets of $A$ such that $A_{\alpha,\emptyset}=A$, $A_{\alpha,s}$ is infinite, $A_{\alpha,s\hat{\;}i}\subset A_{\alpha,s}$ ($i=0,1$) and $A_{\alpha,s\hat{\;}0}\cap A_{\alpha,s\hat{\;}1}=\emptyset$ for every $s$. Then for each $\alpha$ and $s$ we have $A_{\alpha,s}\notin\mathcal{I}$. Mimicing  the proof of Theorem \ref{SpaceabilityFamilyPTheorem} one can show that $\mathcal{I}$ fulfills the assertions of both (i) and (ii) in Theorem \ref{MainTheorem}. 
 
\begin{problem}
    Is there an ideal $\mathcal{I}$ fulfilling the assertion (i) in Theorem \ref{MainTheorem} but not (ii)? 
\end{problem}

\section{Functions in $C(K)$ with interval ranges}\label{Section2Lineability}

In this section we will show that the set $\{f\in C(P_{\mathcal{I}}):\operatorname{rng}f$ is an interval $\}\cup\{0\}$ that can contain a very rich subspace. This is an idealized version of the result of \cite{LRS}, which says that the set of sequences with limit sets of the cardinality continuum contains, up to the zero sequence, an isometric copy of $\ell_\infty$. It should be compared with \cite[Proposition 3.1]{MP}, which says that if a subspace of $\ell_\infty$ is isomorphic to $\ell_\infty$, then it contains a sequence $x$ such that $\operatorname{Lim}(x)$ contains an interval. 

\begin{thm}
Assume that $\mathcal{I}$ has the Baire property. Then the set
\[
\{f\in C(P_{\mathcal{I}}):\operatorname{rng}f\text{ is an interval}\}\cup\{0\}
\]
contains subspace isometric to $C(\beta\N\setminus\N)$.
\end{thm}

\begin{proof}
Let $\mathcal{I}$ be an ideal with the Baire property. Then there is a sequence $n_1<n_2<...$ of indices such that $\bigcup_{i\in K}[n_i,n_{i+1})\notin\mathcal{I}$ for any infinite set $K\subset\N$. Let $\{A_\ell:\ell\in\N\}$ be a partition of $\N$ into infinite sets and let $\{q_\ell:\ell\in\N\}$ be an enumeration of rationals in $[-1,1]$. Define $z\in\ell_\infty$ be the following formula
\[
z(k)=q_\ell\iff k\in\bigcup_{i\in A_\ell}[n_i,n_{i+1}).
\]
Then clearly $\operatorname{L}_{\mathcal{I}}(z)=[-1,1]$.

Now, let $\N\ni n\mapsto(k_n,\ell_n)\in\N\times\N$ be a bijection. Then the formula $F(x)(n)=x(k_n)z(\ell_n)$ defines an isometrical embedding of $\ell_\infty$ into itself. Note that $\operatorname{L}_{\mathcal{I}}(F(x))=[-\Vert x\Vert,\Vert x\Vert]$. Thus the set $\{x\in\ell_\infty:\operatorname{L}_{\mathcal{I}}(x)$ is an interval$\}\cup\{0\}$ contains an isomorphic copy of $\ell_\infty$, and therefore 
$
\{f\in C(P_{\mathcal{I}}):\operatorname{rng}f\text{ is an interval}\}\cup\{0\}
$
contains a subspace isometric to $C(\beta\N\setminus\N)$.
\end{proof}

From the proof of \cite[Proposition 3.1]{MP} one can easily obtain that if $T:\ell_\infty/c_0\to\ell_\infty/c_0$ is an isometrical embedding, then there is $x\in\ell_\infty$ such that the set of limit points of $T([x]_{c_0})$ contains an interval. The fact that $\ell_\infty/c_0$ is isometric to $C(\beta\N\setminus\N)$, implies that very rich space can contain only sets of functions with ranges of cardinality continuum. Therefore we have the following.

\begin{prop}
Let $K$ be a compact space. If $\{f\in C(K):\vert\operatorname{rng}f\vert=\kappa\}\cup\{0\}$ contains an isometric copy of $C(\beta\N\setminus\N)$, then $\kappa=\mathfrak c$. 
\end{prop}

\end{document}